\documentclass[reqno]{amsart}

\newtheorem{theorem}{Theorem}
\newtheorem{lemma}[theorem]{Lemma}
\newtheorem{proposition}[theorem]{Proposition}

\theoremstyle{definition}
\newtheorem{definition}[theorem]{Definition}
\usepackage{graphicx}
\usepackage[all]{xy}
\usepackage{amssymb,amsthm,amsfonts,amsmath,pifont}

\usepackage{pstricks, enumerate, pst-node, pst-text, pst-plot}

\theoremstyle{remark}
\newtheorem{remark}[theorem]{\bf Remark}

\numberwithin{equation}{section}
\numberwithin{theorem}{section}

\newcommand{\intav}[1]{\mathchoice {\mathop{\vrule width 6pt height 3 pt depth  -2.5pt
\kern -8pt \intop}\nolimits_{\kern -6pt#1}} {\mathop{\vrule width
5pt height 3  pt depth -2.6pt \kern -6pt \intop}\nolimits_{#1}}
{\mathop{\vrule width 5pt height 3 pt depth -2.6pt \kern -6pt
\intop}\nolimits_{#1}} {\mathop{\vrule width 5pt height 3 pt depth
-2.6pt \kern -6pt \intop}\nolimits_{#1}}}

\newcommand{\intavl}[1]{\mathchoice {\mathop{\vrule width 6pt height 3 pt depth  -2.5pt
\kern -8pt \intop}\limits_{\kern -6pt#1}} {\mathop{\vrule width 5pt
height 3  pt depth -2.6pt \kern -6pt \intop}\nolimits_{#1}}
{\mathop{\vrule width 5pt height 3 pt depth -2.6pt \kern -6pt
\intop}\nolimits_{#1}} {\mathop{\vrule width 5pt height 3 pt depth
-2.6pt \kern -6pt \intop}\nolimits_{#1}}}

\newcommand{\R}{\mathbb{R}}
\newcommand{\N}{\mathbb{N}}

\newcommand{\Z}{\mathbb{Z}}

\newcommand{\al}{\alpha}
\newcommand{\be}{\beta}
\newcommand{\la}{\lambda}

\newcommand{\ve}{{\varepsilon}}

\begin{document}

\title[A Marstrand theorem for subsets of integers]{A Marstrand theorem for subsets of integers}

\author[Yuri Lima]{Yuri Lima}
\address{Department of Mathematics\\
University of Maryland\\
College Park, MD 20742, USA.}
\email{yurilima@gmail.com}

\author[Carlos Gustavo Moreira]{Carlos Gustavo Moreira}
\address{Instituto Nacional de Matem\'atica Pura e Aplicada, Estrada Dona Castorina 110, 22460-320, Rio de Janeiro, Brasil.}
\email{gugu@impa.br}
\

\subjclass[05A99, 37E15]{05A99, 37E15}

\date{\today}


\begin{abstract}
We propose a counting dimension for subsets of $\Z$ and prove that, under certain conditions
on $E,F\subset\Z$, for Lebesgue almost every $\lambda\in\R$ the counting dimension of
$E+\lfloor\lambda F\rfloor$ is at least the minimum between $1$ and the sum of the counting
dimensions of $E$ and $F$. Furthermore, if the sum of the counting dimensions of $E$ and $F$ is larger than $1$,
then $E+\lfloor\lambda F\rfloor$ has positive upper Banach density for
Lebesgue almost every $\lambda\in\R$. The result has direct consequences when $E,F$ are arithmetic
sets, e.g. the integer values of a polynomial with integer coefficients.
\end{abstract}

\maketitle

\section{Introduction}

The purpose of this paper is to prove a Marstrand type theorem for a class of subsets of the integers.

The well-known theorem of Marstrand \cite{Ma} on geometric measure theory states the following:
if $K\subset\R^2$ is a Borel set
then, for almost every direction, its projection to $\R$ in the respective direction has Hausdorff dimension
equal to the minimum between $1$ and the Hausdorff dimension of $K$; if in addition $K$ has Hausdorff dimension
greater than one, then almost every such projection has positive Lebesgue measure. When $K=K_1\times K_2$ for
$K_1,K_2\subset\mathbb R$, the projections are affine images of the {\it arithmetic sum}
$$K_1+\la K_2=\{x+\la y:x\in K_1,y\in K_2\}$$
and Marstrand's theorem states that $K_1+\la K_2$ has the aforementioned
properties for Lebesgue almost every $\la\in\R$.

The investigation of such {\it arithmetic sums} is an active area of Mathematics,
specially because of its applications in various fields, e.g. diophantine approximations and
dynamical bifurcations.

Given $E\subset\Z$, let $d^*(E)$ denote its {\it upper Banach density}
$$d^*(E)=\limsup_{|I|\rightarrow\infty}\dfrac{|E\cap I|}{|I|}\,,$$
where $I$ runs over all intervals of $\Z$. A remarkable result in
additive combinatorics is Szemer\'edi's theorem \cite{Sz}: if $d^*(E)>0$, then $E$
contains arbitrarily long arithmetic progressions. One can interpret this result by saying that
density represents the correct notion of largeness needed to preserve finite configurations of $\Z$.

Szemer\'edi's theorem does not apply to subsets of zero upper Banach
density. Many of these sets are of interest and they may, as well, contain
combinatorially rich patterns. For example, sets formed by the integer values of a
polynomial with integer coefficients have special interest
in ergodic theory and its connections with combinatorics \cite{BL,Bo}.
Another example are the prime numbers: they have zero density (by the prime number theorem),
and yet there are arbitrarily long arithmetic progressions of primes \cite{GT}.

A set $E\subset\Z$ of zero upper Banach density occupies portions in intervals of $\Z$
that grow sublinearly as the length of the intervals grow, but there still may exist
some sublinear growth speed, e.g. the number of perfect squares on $(0,n]$ is
about $n^{0.5}$. This exponent represents, in some sense, a dimension of $\{n^2:n\in\Z\}$ inside $\Z$.
In this article, we propose a {\it counting dimension}
$$D(E)=\limsup_{|I|\rightarrow\infty}\dfrac{\log{|E\cap I|}}{\log{|I|}}\,,$$
where $I$ runs over all intervals of $\Z$. This definition captures the growth rate of $|E\cap I|$, and it
allows us to compare largeness between sets of zero upper Banach density.

Up to now, a theory of {\it fractal sets} in $\Z$ has not been developed, and this is the motivation for
this article. We address the following question: to what extent fractal sets in $\Z$ satisfy
a Marstrand like theorem? By fractal set we mean a set $E\subset\Z$ with $D(E)<1$. For example,
if $p\in\Z[x]$ has degree $d>0$, then $E=\{p(n)\}_{n\in\N}$ has counting dimension $\frac{1}{d}$
(see Section~\ref{sub polynomials}). The first main result of this paper is that such
{\it polynomial sets} satisfy a Marstrand like theorem.

\begin{theorem}\label{thm polynomial}
Let $p_i\in\Z[x]$ with degree $d_i>0$, and let $E_i=\{p_i(n)\}_{n\in\Z}$. Then
$$D(E_0+\lfloor\la_1 E_1\rfloor+\cdots+\lfloor\la_k E_k\rfloor)\ge\min
\left\{1,\dfrac{1}{d_0}+\dfrac{1}{d_1}+\cdots+\dfrac{1}{d_k}\right\}$$
for Lebesgue almost every $\la=(\la_1,\ldots,\la_k)\in\R^k$. If $\sum_{i=0}^k\frac{1}{d_i}>1$,
then $E_0+\lfloor\la_1 E_1\rfloor+\cdots+\lfloor\la_k E_k\rfloor$ has positive
upper Banach density for Lebesgue almost every $\la=(\la_1,\ldots,\la_k)\in\R^k$.
\end{theorem}

Theorem \ref{thm polynomial} is consequence of a more general result, which is the main result of this paper.
It identifies sufficient conditions for a Marstrand theorem on $\Z$ to hold.

\begin{theorem}\label{main thm 1}
Let $E,F\subset\Z$ be regular compatible sets. Then
$$D(E+\lfloor\lambda F\rfloor)\ge\min\{1,D(E)+D(F)\}$$
for Lebesgue almost every $\lambda\in\R$. If $D(E)+D(F)>1$, then
$E+\lfloor\lambda F\rfloor$ has positive upper Banach density for Lebesgue almost
every $\lambda\in\R$.
\end{theorem}

Let us briefly explain the notions of regularity and compatibility. The counting dimension
says that $|E\cap I|$, along a sequence of intervals $I$, grows like $|I|^{D(E)}$,
with a small error on the exponent $D(E)$. We say that $E$ is {\it regular} when there is
no error at all, i.e. when the cardinality of $|E\cap I|$ is, up to a multiplicative constant,
of the order of $|I|^{D(E)}$. See Section~\ref{sub regular sets} for the definition.

Now let $E,F$ be two regular sets: there are intervals $I$ such that $|E\cap I|$ has order $|I|^{D(E)}$
and intervals $J$ such that $|F\cap J|$ has order $|J|^{D(F)}$. In general, $|I|$ and $|J|$ are incomparable.
We say that $E$ and $F$ are {\it compatible} if $|I|$ and $|J|$ are asymptotic.
See Section~\ref{sub regular sets} for the definition.

The quantities $d^*(E)$ and $D(E)$ are similar to the
Lebesgue measure and box dimension on $\R$. It is because of this association that
we call Theorem \ref{main thm 1} a Marstrand theorem for subsets of integers.
Most results of this paper were motivated by known facts in geometric measure theory.
We will try to refer to these facts.

The notions of regularity and compatibility are both satisfied by many arithmetic
subsets of $\Z$, e.g. the integer values of a polynomial with integer coefficients and,
more generally, by {\it universal} sets: those sets that exhibit the expected growth rate along
intervals of arbitrary length (see Definition~\ref{def universal}). For these sets,
Theorem~\ref{main thm 1} can be inductively applied to give the result below.

\begin{theorem}\label{main thm 2}
Let $E_0,\ldots,E_k$ be universal subsets of $\Z$. Then
$$D(E_0+\lfloor\la_1 E_1\rfloor+\cdots+\lfloor\la_k E_k\rfloor)\ge\min
\left\{1,D(E_0)+D(E_1)+\cdots+D(E_k)\right\}$$
for Lebesgue almost every $\la=(\la_1,\ldots,\la_k)\in\R^k$. If $\sum_{i=0}^k D(E_i)>1$,
then $E_0+\lfloor\la_1 E_1\rfloor+\cdots+\lfloor\la_k E_k\rfloor$ has positive
upper Banach density for Lebesgue almost every $\la=(\la_1,\ldots,\la_k)\in\R^k$.
\end{theorem}

Integer values of a polynomial with integer coefficients are universal sets, thus
Theorem \ref{thm polynomial} follows from Theorem \ref{main thm 2}.

The proof of Theorem \ref{main thm 1} is based on the ideas developed in \cite{LM,LM2}.
The cardinality of a regular subset of $\Z$
along an increasing sequence of intervals exhibits an exponential behavior ruled out by its counting
dimension. If this holds for two regular subsets $E,F\subset\Z$, the compatibility assumption allows
to estimate the cardinality of $E+\lfloor\la F\rfloor$ along the respective arithmetic
sums of intervals. A double-counting argument estimates the size of the ``bad'' parameters for
which such cardinality is small.

The paper is organized as follows. In \S\ref{section preliminaries} we provide basic notations
and definitions. In \S\ref{section examples} we discuss some examples, including the sets given by
integer values of a polynomial with integer coefficients.
In \S\ref{section regularity and compatibility} we introduce the notions
of regularity and compatibility. In \S\ref{sub counterexample} we construct a counterexample to
Theorem~\ref{main thm 1} when the sets are not compatible (thus regularity and compatibility
are not only sufficient but also necessary conditions for the validity of Theorem~\ref{main thm 1}).
In \S\ref{sub counterexample 2} we construct a counterexample to Theorem~\ref{main thm 1} when the space of
parameters is $\Z$ (thus $\R$ is the correct space of parameters).
In \S\ref{section marstrand thm} we prove Theorems~\ref{main thm 1} and \ref{main thm 2}.
We also collect some final remarks and questions in \S\ref{section final remarks}.

\section{Preliminaries}\label{section preliminaries}

\subsection{General notation}

Given a set $X$, $|X|$ denotes the cardinality of $X$.
$\Z$ denotes the set of integers and $\N$ the set of positive integers.

\begin{definition}\label{def vinogradov}
Let $f,g:\Z$ or $\N\rightarrow\R$. We write $f\lesssim g$ if there is $C>0$ such that
$$|f(n)|\le C|g(n)|\,,\ \ \forall\,n\in\Z\text{ or }\N.$$
If $f\lesssim g$ and $g\lesssim f$, we write $f\sim g$. We write $f\approx g$ if
$$\lim_{|n|\rightarrow\infty}\dfrac{f(n)}{g(n)}=1.$$
\end{definition}

Given $x\in\R$, $\lfloor x\rfloor$ is the integer part of $x$. For $k\ge 1$, $m_k$
is the Lebesgue measure of $\R^k$. Let $m=m_1$. The letter $I$ denotes
an interval of $\Z$, e.g. $I=(M,N]=\{M+1,\ldots,N\}.$
The {\it length} of $I$ is $|I|=N-M$.

Given $E\subset\Z$ and $\la\in\R$, let $\la E=\{\la n:n\in E\}\subset\R$ and
$\lfloor\la E\rfloor=\{\lfloor\la n\rfloor:n\in E\}\subset\Z$.

\subsection{Counting dimension}\label{sub definitions}

\begin{definition}\label{def density}
The {\it upper Banach density} of $E\subset\Z$ is
$$d^*(E)=\limsup_{|I|\rightarrow\infty}\dfrac{|E\cap I|}{|I|}\,,$$
where $I$ runs over all intervals of $\Z$.
\end{definition}

\begin{definition}\label{def dimension}
The {\it counting dimension} or simply {\it dimension} of $E\subset \Z$ is
$$D(E)=\limsup_{|I|\rightarrow\infty}\dfrac{\log{|E\cap I|}}{\log{|I|}}\,,$$
where $I$ runs over all intervals of $\Z$.
\end{definition}

The above definition is similar to the box dimension on $\mathbb R$.
Similar definitions appeared in \cite{BF,Fu}.
Note that $D(E)\in[0,1]$, and if $d^*(E)>0$ then $D(E)=1$.

Here is an alternative definition of $D(E)$ that is
similar in spirit to the Hausdorff dimension on $\R$.
Let $\al$ be a nonnegative real number.

\begin{definition}\label{def alpha-dimension}
The {\it counting $\al$-measure} of $E\subset\Z$ is
$$H_\al(E)=\limsup_{|I|\rightarrow\infty}\dfrac{|E\cap I|}{|I|^\al}\,,$$
where $I$ runs over all intervals of $\Z$.
\end{definition}

Clearly, $H_\al(E)\in[0,\infty]$. For a fixed $E\subset\Z$, the numbers $H_\al(E)$ are decreasing
in $\al$. Furthermore,
$$\al<D(E)\ \Longrightarrow\ H_\al(E)=\infty\ \ \text{ and }\ \ \al>D(E)\ \Longrightarrow\ H_\al(E)=0\,,$$
thus there is a unique $\al\ge 0$ such that
$$
H_\be(E)=\left\{
\begin{array}{ll}
\infty&,\text{if }0\le \be<\al,\\
0&,\text{if }\be>\al.
\end{array}\right.
$$
Thus $D(E)=\al$, i.e. $D(E)$ is the parameter
$\al$ where $H_\al(E)$ decreases from infinity to zero.

Here is an analogue to Frostman's lemma (see Theorem 8.8 of \cite{Mattila}):
if $\be>D(E)$, then
\begin{equation}\label{eq 21}
|E\cap I|\lesssim |I|^\be,
\end{equation}
where $I$ runs over all intervals of $\Z$. Conversely, if (\ref{eq 21}) holds, then $D(E)\le\be$.

Below we collect some basic properties of $D$ and $H_\al$. All proofs are direct.
\begin{enumerate}[(i)]
\item If $E\subset F$, then $D(E)\le D(F)$.
\item $D(E\cup F)=\max\{D(E),D(F)\}$.
\item If $\la\ge 1$, then
\begin{align}\label{eq 14}
H_\al(\lfloor\la E\rfloor)=\la^{-\al}H_\al(E).
\end{align}
\end{enumerate}

\begin{remark}
$\lfloor -x\rfloor=-\lfloor x\rfloor$ or $\lfloor -x\rfloor=-\lfloor x\rfloor-1$, thus
$D(\lfloor-\la E\rfloor)=D(\lfloor\la E\rfloor)$. Also,
$0<H_\al(\lfloor-\la E\rfloor)<\infty$ iff $0<H_\al(\lfloor\la E\rfloor)<\infty$. So we
assume, from now on, that $\la>0$.
\end{remark}

\section{Examples}\label{section examples}

\noindent{\bf Example 1.} Let $\al\in(0,1]$, and let
$E_\al=\left\{\left\lfloor n^{1/\al}\right\rfloor:n\in\N\right\}$.
We claim that $H_\al(E_\al)=1$. Because\footnote{For each $t\ge 0$, the function $x\in[0,t]\mapsto x^\al+(t-x)^\al$
is concave, so it attains its minima at $x=0$ and $x=t$. Thus $x^\al+(t-x)^\al\ge t^\al$ for
any $x\in[0,t]$.}
$(x+y)^\al\le x^\al+y^\al$ for $x,y\ge 0$, we have
$$\dfrac{|E_\al\cap(M,N]|}{(N-M)^\al}\le\dfrac{(N+1)^\al-(M+1)^\al}{(N-M)^\al}\le 1,$$
thus $H_\al(E_\al)\le 1$. On the other hand, $\frac{|E_\al\cap(0,N]|}{N^\al}\ge\frac{N^\al-1}{N^\al}$,
thus $H_\al(E_\al)\ge 1$.\\

\noindent{\bf Example 2.} The prime numbers have dimension one. This follows from the prime number theorem:
$$\lim_{n\rightarrow\infty}\dfrac{\log|\{1\le p\le n:p\text{ is prime}\}|}{\log n}=
\lim_{n\rightarrow\infty}\dfrac{\log n-\log\log n}{\log n}=1.$$

\subsection{Polynomial subsets of $\Z$}\label{sub polynomials}

\begin{definition}
A {\it polynomial set of $\Z$} is a set $E=\{p(n):n\in\N\}$, where $p(x)\in\Z[x]$ is non-constant.
\end{definition}

These are the sets we consider in Theorem \ref{thm polynomial}. Let us calculate their counting dimension.
Call $E,F\subset\Z$ {\it asymptotic} if $E=\{\cdots<a_{-1}<a_0<a_1<\cdots\}$,
$F=\{\cdots<b_{-1}<b_0<b_1<\cdots\}$ and there is $i\ge0$ such that
\begin{equation}\label{eq 16}
a_{n-i}\le b_n\le a_{n+i}\ ,\ \ \text{for all }n\in\Z.
\end{equation}

\begin{lemma}\label{lemma asymptotic sets}
If $E,F$ are asymptotic and $\al>0$, then $H_\al(E)=H_\al(F)$. In particular, $D(E)=D(F)$.
\end{lemma}

\begin{proof}
Let $I=(M,N]$, and let $E\cap I=\{a_{m+1},a_{m+2},\ldots,a_n\}$. By (\ref{eq 16}),
$$\{b_{m+i+1},\ldots,b_{n-i}\}\subset F\cap I\subset\{b_{m-i},\ldots,b_{n+i+1}\},$$
thus $|E\cap I|\approx |F\cap I|$, and so $H_\al(E)=H_\al(F)$.
\end{proof}

Let $E=\{p(n):n\in\N\}$, where $p$ has degree $d$. Assume that $p$ has leading coefficient $a\ge 1$.
Thus there is $i\ge 0$ such that $a(n-i)^d<p(n)<a(n+i)^d$ for sufficiently large $n\in\Z$, so
$E,aE_{\frac{1}{d}}$ are asymptotic, where $E_{\frac{1}{d}}$ is defined as in Example 1. By equality
(\ref{eq 14}) and Lemma \ref{lemma asymptotic sets}, it follows that $D(E)=\frac{1}{d}$ and
that $H_{\frac{1}{d}}(E)=a^{-\frac{1}{d}}$.

\subsection{Cantor sets in $\Z$}\label{sub Cantor}

The classical ternary Cantor set of $\R$ is the set of real numbers on $[0,1]$ with only $0$'s and
$2$'s on the expansion in base 3. In analogy to this, define $E\subset\Z$ as
\begin{equation}\label{eq 20}
E=\left\{\sum_{i=0}^n a_i3^i:n\in\N\text{ and }a_i\in\{0,2\}\right\}.
\end{equation}
Fisher proved in \cite{Fi} that $H_{\frac{\log 2}{\log 3}}(E)>0$. In Lemma \ref{lemma dim=perron} below,
we will prove that $H_{\frac{\log 2}{\log 3}}(E)<\infty$. In particular, $D(E)=\frac{\log 2}{\log 3}$.

The renormalization of $E\cap(0,3^n)$ via the map $x\mapsto\frac{x}{3^n}$ is a subset of $(0,1)$
equal to the set of left endpoints of the remaining intervals of the $n$-th step of the construction of the
classical ternary Cantor set of $\R$, i.e. if $K=\bigcup_{n\in E}[n,n+1]$, then $\frac{K}{3^n}$ is the
$n$-th step of the construction of the ternary Cantor set of $\R$.

More generally, let us define a class of Cantor sets in $\Z$. Fix $a\in\N$ and a
binary matrix $A=(a_{ij})_{0\le i,j\le a-1}$. For $n\ge 0$, let
$$\Sigma_n(A)=\left\{(d_0d_1\cdots d_n):a_{d_{i-1}d_i}=1,\ 1\le i\le n\right\},$$
and let $\Sigma^*(A)=\bigcup_{n\ge 0}\Sigma_n(A)$.

\begin{definition}\label{def cantor set}
The {\it integer Cantor set} $E_A\subset\Z$ induced by $A$ is
$$E_A=\{d_0a^0+\cdots+d_na^n:(d_0d_1\cdots d_n)\in\Sigma^*(A)\}.$$
\end{definition}

Here is our motivation for Definition \ref{def cantor set}: dynamically defined topologically mixing Cantor sets
of the real line are homeomorphic to subshifts of finite type (see e.g. \cite{LM}). After truncating the numbers,
this is exactly what Definition \ref{def cantor set} does.

Remember that the {\it Perron-Frobenius eigenvalue} of $A$ is its largest eigenvalue $\la_+(A)$.
It has multiplicity one and maximizes the absolute value of the eigenvalues of $A$. Furthermore,
there is $c=c(A)>0$ such that
\begin{equation}\label{eq 17}
c^{-1}{\la_+(A)}^n\le|\Sigma_n(A)|\le c{\la_+(A)}^n\,, \ \ \text{for all }n\ge 0.
\end{equation}
See e.g. \cite{Kit}. The dimension of $E_A$ depends explicitly on $a$ and on $\la_+(A)$.

\begin{lemma}\label{lemma dim=perron}
If $A$ is a binary $a\times a$ matrix, then
$$D(E_A)=\dfrac{\log \la_+(A)}{\log a}\ \ \text{ and }\ \ 0<H_{\frac{\log\la_+(A)}{\log a}}(E_A)<\infty\,.$$
\end{lemma}

\begin{proof}
Let $I=(M,N]$. We may assume that $M+1,N\in E_A$, say
\begin{align*}
\left\{
\begin{array}{rcl}
M+1&=&x_0a^0+\cdots+x_na^n\\
N&=&y_0a^0+\cdots+y_na^n
\end{array}\right.,
\end{align*}
where $y_n>x_n$ (if $x_n=y_n$, then we can consider the translation of $I$ by $-x_na^n$).
If $y_n\ge x_n+2$, then
$$
\left\{\begin{array}{rcl}
M+1&\le&(x_n+1)a^n\\
N&\ge & (x_n+2)a^n
\end{array}\right.
\Longrightarrow\ |I|\ge a^n.
$$
Because $I\subset(0,a^{n+1})$, we get that
\begin{equation}\label{eq 18}
\dfrac{|E_A\cap I|}{|I|^{\frac{\log\la_+(A)}{\log a}}}\le\dfrac{|\Sigma_n(A)|}{a^{\frac{n\log\la_+(A)}{\log a}}}\le\dfrac{c{\la_+(A)}^n}{{\la_+(A)}^n}=c\,.
\end{equation}
If $y_n=x_n+1$, let $i,j\in\{0,1,\ldots,n-1\}$ such that
\begin{enumerate}[(i)]
\item $x_i<a-1$ and $x_{i+1}=\cdots=x_{n-1}=a-1$,
\item $y_j>0$ and $y_{j+1}=\cdots=y_{n-1}=0$.
\end{enumerate}
Thus
$$
\left\{\begin{array}{rcl}
M+1&\le& (x_n+1)a^n-a^i\\
N&\ge &(x_n+1)a^n+a^j
\end{array}\right.
\Longrightarrow\ |I|\ge a^i+a^j\ge a^{\max\{i,j\}}.
$$
If $\sum_{l=0}^n z_la^l\in I$, then necessarily $z_n\in\{x_n,x_n+1\}$. In the
first case $z_{i+1}=\cdots=z_{n-1}=a-1$, and in the second case $z_{j+1}=\cdots=z_{n-1}=0$. Thus
$$|E_A\cap I|\le |\Sigma_i(A)|+|\Sigma_j(A)|\le 2c{\la_+(A)}^{\max\{i,j\}},$$
and so
\begin{equation}\label{eq 19}
\dfrac{|E_A\cap I|}{|I|^{\frac{\log\la_+(A)}{\log a}}}\le
\dfrac{2c{\la_+(A)}^{\max\{i,j\}}}{a^{\frac{\max\{i,j\}\log\la_+(A)}{\log a}}}=2c\,.
\end{equation}
Estimates (\ref{eq 18}) and (\ref{eq 19}) give that $H_{\frac{\log\la_+(A)}{\log a}}(E_A)<\infty$. Furthermore,
$$\dfrac{|E_A\cap(0,a^n]|}{a^{\frac{n\log\la_+(A)}{\log a}}}\ge \dfrac{c^{-1}{\la_+(A)}^{n-1}}{{\la_+(A)}^n}=
c^{-1}\la_+(A)^{-1},$$
thus $H_{\frac{\log\la_+(A)}{\log a}}(E_A)>0$.
\end{proof}

Here is a direct application of Lemma \ref{lemma dim=perron}: if $X\subset\{0,\ldots,a-1\}$
and $A=(a_{ij})$ with $a_{ij}=1$ iff $i,j\in X$, then $D(E_A)=\frac{\log{|X|}}{\log a}$.

If $E,F\subset\Z$ with $D(E)+D(F)>1$, then it is not true in general that $d^*(E+F)>0$. This happens
if the elements of $E+F$ have many representations as the sum of one element of $E$ and other
of $F$. This resonance phenomenon might decrease the dimension of $E+F$. Fore example:
if $E=E_A$ and $F=E_B$, where $A=(a_{ij})_{0\le i,j\le 11}$,
$B=(b_{ij})_{0\le i,j\le 11}$ are defined by
$$a_{ij}=1\,\iff\,0\le i,j\le 3\ \ \text{ and }\ \ b_{ij}=1\,\iff\,4\le i,j\le 7,$$
then $D(E)+D(F)=\frac{2\log 4}{\log 12}$, while $E+F=E_C$ for $C=(c_{ij})_{0\le i,j\le 11}$ defined by
$$c_{ij}=1\,\iff\,4\le i,j\le 10.$$
$E+F$ has counting dimension equal to $\frac{\log 7}{\log 12}$, thus $d^*(E+F)=0$. What Theorem \ref{main thm 1}
gives is that resonance is avoided if we change the scales of the sets.

\subsection{Generalized IP-sets}\label{sub gap}

This class of sets was suggested to us by Simon Griffiths and Rob Morris.
Let $(k_n)_{n\ge 1}$, $(d_n)_{n\ge 1}$ be sequences of positive integers with $k_n\ge 2$.

\begin{definition}
The {\it generalized IP-set} associated to $(k_n)_{n\ge 1}$, $(d_n)_{n\ge 1}$ is the set
$$E=\left\{\sum_{i=1}^n x_id_i:n\in\N\text{ and }0\le x_i< k_i\right\}.$$
\end{definition}

We always assume that $d_n>\sum_{i=1}^{n-1}k_id_i$.
Thus the map $\sum_{i=1}^n x_id_i\mapsto (x_1,\ldots,x_n)$ is a bijection
from $E$ to $\left\{(x_1,\ldots,x_n):n\in\N, x_n>0\text{ and }\,0\le x_i< k_i\right\}$.
Also, if this former set is colexicographically ordered\footnote{The sequence $(x_1,\ldots,x_n)$
is smaller than $(y_1,\ldots,y_m)$ if $n<m$ or if there is $i\in\{1,\ldots,n\}$ such that
$x_i<y_i$ and $x_j=y_j$ for $j=i+1,\ldots,n$.}, then the map is order-preserving.

\begin{lemma}\label{lemma gap}
Let $E$ be the generalized IP-set associated to $(k_n)_{n\ge 1}$, $(d_n)_{n\ge 1}$,
and let $p_n=k_1\cdots k_n$. Then
\begin{equation}\label{eq 34}
D(E)=\limsup_{n\rightarrow\infty}\dfrac{\log p_n}{\log{k_nd_n}}\cdot
\end{equation}
\end{lemma}

\begin{proof}
Firstly, note that if $I=(0,\sum_{i=1}^n(k_i-1)d_i]$, then $|E\cap I|=p_n$ and $|I|\le k_nd_n$. Thus
$D(E)\ge\limsup_{n\rightarrow\infty}\dfrac{\log p_n}{\log{k_nd_n}}$.
Now let $I=(M,N]$, say
$$
\left\{\begin{array}{rcl}
M+1&=&x_1d_1+\cdots+x_nd_n\\
N&=&y_1d_1+\cdots+y_nd_n.
\end{array}\right.
$$
As in Lemma \ref{lemma dim=perron}, we can assume that $y_n>x_n$. Let $y_n=x_n+k$, where $0<k<k_n$.
We divide the analysis into three cases:\\

\noindent{\bf Case 1.} $M=0$: we have $|E\cap I|\le (k+1)p_{n-1}$ and $|I|\ge kd_n$.
Because the map $k\in\mathbb N\mapsto\frac{\log{(k+1)p_{n-1}}}{\log{kd_n}}$ is increasing,
we get that
\begin{align*}
\dfrac{\log{|E\cap I|}}{\log{|I|}}\le\dfrac{\log{(k_n+1)p_{n-1}}}{\log{k_nd_n}}=
(1+o(1))\dfrac{\log{k_np_{n-1}}}{\log{k_nd_n}}=(1+o(1))\dfrac{\log{p_n}}{\log{k_nd_n}}\cdot
\end{align*}

\noindent{\bf Case 2.} $k\ge 2$: we have $|E\cap I|\le(k+1)p_{n-1}$
and $|I|\ge (k-1)d_n$, thus we can proceed as in Case 1.\\

\noindent{\bf Case 3.} $k=1$: Let $P=\sum_{i=1}^{n-1}(k_i-1)d_i+x_nd_n$,
and let $I_1=(M,P]$ and $I_2=(P,N]$. Thus $E\cap I_2=P+E\cap(0,N-P]$, so $|E\cap I_2|,|I_2|$ can be estimated
as in Case 1. Similarly, we estimate $|E\cap I_1|,|I_1|$ (just consider a reflected version of $E$ and
apply Case 1).
Note that either $|I_1|\to+\infty$ or $|I_2|\to+\infty$.\\

\noindent Thus $D(E)\le\limsup_{n\rightarrow\infty}\dfrac{\log p_n}{\log{k_nd_n}}$.
\end{proof}

\section{Regularity and compatibility}\label{section regularity and compatibility}

\subsection{Regular sets}\label{sub regular sets}

\begin{definition}
We call $E\subset\Z$ {\it regular} or $\alpha$-{\it set} if $D(E)=\alpha$ and $0<H_\alpha(E)<\infty$.
\end{definition}

By Lemmas \ref{lemma asymptotic sets} and \ref{lemma dim=perron}, polynomial sets and Cantor sets are regular.
A general set is not regular, but it does contain many regular subsets.

\begin{proposition}\label{prop smaller dimension}
Let $E\subset\Z$ and $0\le\al\le 1$. If $H_\al(E)>0$, then there is a regular subset $E'\subset E$
such that $D(E')=\al$. In particular, if $0\le\al< D(E)$ then there is $E'\subset E$ regular such that
$D(E')=\al$.
\end{proposition}

\begin{proof}
This is an analogue of Theorem 8.19 of \cite{Mattila}, and the idea of the proof is similar in spirit:
we apply a dyadic argument to decrease $H_\al(E)$ in a controlled way.
Given an interval
$I\subset\Z$ and a subset $F\subset\Z$, define
$$s_F(I)\doteq\sup_{J\subset I\atop{J\text{ interval}}}\dfrac{|F\cap J|}{|J|^\al}\,\cdot$$
If $F=\{a_1,a_2,\ldots,a_k\}\subset I$, the operation of alternately discard the elements of $F$,
$$F=\{a_1,a_2,\ldots,a_k\}\rightsquigarrow F'=\{a_1,a_3,a_5,\ldots,a_{2\left\lceil\frac{k-1}{2}\right\rceil-1},a_k\},$$
decreases $s_F(I)$ to approximately $\frac{s_F(I)}{2}$. More specifically,
if $s_F(I)>2$, then $\frac{1}{2}<s_{F'}(I)<s_F(I)-\frac{1}{2}$:
\begin{enumerate}[$\bullet$]
\item For every interval $J\subset I$, it holds
$$\dfrac{|F'\cap J|}{|J|^\al}\le \dfrac{1}{2}\cdot\dfrac{|F\cap J|+1}{|J|^\al}\le\dfrac{s_F(I)}{2}+\dfrac{1}{2}<
s_F(I)-\dfrac{1}{2}\cdot$$
\item If $J$ maximizes $s_F(I)$, then
$$\dfrac{|F'\cap J|}{|J|^\al}\ge\dfrac{1}{2}\cdot\dfrac{|F\cap J|-1}{|J|^\al}>
1-\dfrac{1}{2|J|^\al}\ge\dfrac{1}{2}\,\cdot$$
\end{enumerate}
After a finite number of these operations, we get $F'\subset F$ with
$\frac{1}{2}<s_{F'}(I)\le 2$.

If $H_\al(E)<\infty$, there is nothing to do. Assume that $H_\al(E)=\infty$. We inductively
construct a sequence $F_1\subset F_2\subset\cdots$ of finite subsets of $E$ such that
$F_n\subset I_n=(a_n,b_n]$ with $|I_n|\to+\infty$ and:
\begin{enumerate}[(i)]
\item $\frac{1}{2}<s_{F_n}(I_n)\le 3$.
\item There is an interval $J_n\subset I_n$ such that $|J_n|\ge n$ and
$\frac{|F_n\cap J_n|}{|J_n|^\al}>\frac{1}{2}$.
\end{enumerate}
Once these properties hold, $E'=\bigcup_{n\ge 1}F_n$ will satisfy the required conditions.

Take any $a\in E$ and $I_1=\{a\}$. Assume $I_n,F_n,J_n$ are defined and satisfy (i), (ii).
Because $H_\al(E)=\infty$, there exists an interval $J_{n+1}$ disjoint from
$(a_n-|I_n|^{\frac{1}{\al}},b_n+|I_n|^{\frac{1}{\al}}]$ such that
$$\dfrac{|E\cap J_{n+1}|}{|J_{n+1}|^\al}\ge (n+1)^{1-\al}\,.$$
Thus we can restrict $J_{n+1}$ to a smaller interval of size
at least $n+1$, also denoted $J_{n+1}$, such that
\begin{equation}
s_E(J_{n+1})=\dfrac{|E\cap J_{n+1}|}{|J_{n+1}|^\al}\,\cdot
\end{equation}
Consider $F_{n+1}'=E\cap J_{n+1}$ and apply the dyadic operation to $F_{n+1}'$ until
\begin{equation}\label{eq 22}
\dfrac{1}{2}<s_{F_{n+1}'}(J_{n+1})=\dfrac{|F_{n+1}'\cap J_{n+1}|}{|J_{n+1}|^\al}\le 2.
\end{equation}
Let $I_{n+1}=I_n\cup K_n\cup J_{n+1}$ be the convex hull
of $I_n$ and $J_{n+1}$, and let $F_{n+1}=F_n\cup F_{n+1}'$. Condition (ii) is satisfied because of
(\ref{eq 22}). To prove (i), let $I$ be a subinterval of $I_{n+1}$. We have three cases.
\begin{enumerate}[$\bullet$]
\item $I\subset I_n\cup K_n$: by (i),
$$\dfrac{|F_{n+1}\cap I|}{|I|^\al}\le\dfrac{|F_n\cap(I\cap I_n)|}{|I\cap I_n|^\al}\le 3.$$
\item $I\subset K_n\cup J_{n+1}$: by (\ref{eq 22}),
$$\dfrac{|F_{n+1}\cap I|}{|I|^\al}\le\dfrac{|F_{n+1}'\cap(I\cap J_{n+1})|}{|I\cap J_{n+1}|^\al}\le 2.$$
\item $I\supset K_n$: because $|K_n|\ge |I_n|^{\frac{1}{\al}}$,
\begin{eqnarray*}
\dfrac{|F_{n+1}\cap I|}{|I|^\al}&=&\dfrac{|F_n\cap(I\cap I_n)|+
                                    |F_{n+1}'\cap(I\cap J_{n+1})|}{(|I\cap I_n|+|K_n|+|I\cap J_{n+1}|)^\al}\\
                                &\le&\dfrac{|F_n\cap(I\cap I_n)|}{(|I\cap I_n|+|K_n|)^\al}+
                                    \dfrac{|F_{n+1}'\cap(I\cap J_{n+1})|}{(|I\cap J_{n+1}|)^\al}\\
                                &\le&\dfrac{|I_n|}{|K_n|^{\al}}+s_{F_{n+1}'}(J_{n+1})\\
                                &\le&3.
\end{eqnarray*}
\end{enumerate}
This proves (i) and completes the inductive step.
\end{proof}

\subsection{Compatible sets}\label{sub compatible sets}

\begin{definition}\label{def compatible}
We call two regular subsets $E,F\subset\Z$ {\it compatible} if there are sequences $(I_n)_{n\ge 1}$,
$(J_n)_{n\ge 1}$ of intervals with increasing lengths such that
\begin{enumerate}[(i)]
\item $|I_n|\sim |J_n|$.
\item $|E\cap I_n|\gtrsim |I_n|^{D(E)}$ and $|F\cap J_n|\gtrsim |J_n|^{D(F)}$.
\end{enumerate}
\end{definition}

In other words, two subsets are compatible if they have intervals of comparable lengths such that
their cardinality on these intervals have the correct asymptotic.

\begin{definition}\label{def universal}
We call a regular subset $E\subset\Z$ {\it universal} if there is a sequence $(I_n)_{n\ge 1}$ of
intervals such that $|I_n|\sim n$ and $|E\cap I_n|\gtrsim |I_n|^{D(E)}$.
\end{definition}

Each $E_\al$ is universal, as well as each polynomial sets (see Section \ref{sub polynomials}).
If $E$ is universal and $F$ is regular, then $E,F$ are compatible.
In particular, any two polynomial sets are compatible.

\subsection{A counterexample to Theorem \ref{main thm 1} for regular non-compatible sets}\label{sub counterexample}

Let us show that compatibility is a necessary assumption for Theorem \ref{main thm 1}:
we construct regular sets $E,F\subset\Z$ such that $D(E)+D(F)>1$ and
$E+\lfloor\la F\rfloor$ has zero upper Banach density for every $\la\in\R$. The idea is to construct $E$ and
$F$ such that the intervals $I,J\subset\Z$ on which
$\frac{|E\cap I|}{|I|^{D(E)}}$ and $\frac{|F\cap J|}{|J|^{D(F)}}$ are bounded away from zero have
incomparable lengths.

Let $\al\in\left(\frac{1}{2},1\right)$, and let
$$E=\bigcup_{i\text{ odd}}(E_i\cap I_i)\ \ ,\ \ F=\bigcup_{i\text{ even}}(E_i\cap I_i)$$
such that the following conditions hold:
\begin{enumerate}[(i)]
\item $E_i=\lfloor\mu_i \lfloor {\mu_i}^{-1}E_0\rfloor\rfloor$, where $E_0\subset\N$ with $H_\al(E_0)=\frac{1}{2}$.
\item $\mu_i>b_{i-1}^{\frac{2}{1-\al}}$ for all $i\ge 1$.
\item $(I_i)_{i\ge 1}$ is a disjoint sequence of intervals of increasing lengths such that
$$\lim_{i\rightarrow\infty}\dfrac{|E_i\cap I_i|}{|I_i|^\al}=\dfrac{1}{2}\,\cdot$$
\end{enumerate}
It is clear that we can inductively construct $(\mu_i)_{i\ge 1}, (E_i)_{i\ge 1}$,
$(I_i)_{i\ge 1}$ such that $I_n=(a_n,b_n]$ and $0<b_i<a_{i+1}$. Let us explain conditions (i)--(iii).
(i) gives an $\al$-regular set such that the gap between consecutive elements is at least (of the order of) $\mu_i$.
(ii) implies that $E,F$ are incompatible, and also that the left endpoint of $I_i$ is much
larger than the right endpoint of $I_{i-1}$. (iii) implies that $E,F$ are $\al$-sets.

\begin{lemma}\label{lemma removal}
Let $E,F\subset\Z$ with $D(E),D(F)<1$, $A,B\subset\Z$ finite and $E'=E\cup A$, $F'=F\cup B$. Then
$d^*(E'+\lfloor \la F'\rfloor)=d^*(E+\lfloor \la F\rfloor)$ for all $\la\in\R$.
\end{lemma}

\begin{proof}
$E'+\lfloor \la F'\rfloor=(E+\lfloor \la F\rfloor)\cup (A+\lfloor \la F\rfloor)\cup (E+\lfloor \la B\rfloor)\cup
(A+\lfloor \la B\rfloor)$,
and each of the sets $A+\lfloor \la F\rfloor$, $E+\lfloor \la B\rfloor$, $A+\lfloor \la B\rfloor$
has dimension smaller than one.
\end{proof}

If we fix $\la>0$, then for $i,j$ large enough it holds:
\begin{enumerate}[(i)]
\item[(iv)] $\lfloor\la I_i\rfloor\cap \lfloor\la I_j\rfloor=\emptyset$.
\item[(v)] $b_i>\max\{4\la^{-1},4\la\}^{\frac{1-\al}{1+\al}}$.
\end{enumerate}
By Lemma \ref{lemma removal}, we can delete the intervals $I_i$ for small $i$'s without changing
$d^*(E+\lfloor\la F \rfloor)$. Thus we can assume (iv), (v) for all $i,j$. Under these assumptions,
$$(I_i+\lfloor\la I_j\rfloor)\cap(I_k+\lfloor\la I_l\rfloor)\not=\emptyset\ \iff\ i=k\text{ and }j,l<i\
\text{ or }\ j=l\text{ and }i,k<j.$$
This follows from (ii): if $(I_i+\lfloor\la I_j\rfloor)\cap(I_k+\lfloor\la I_l\rfloor)\not=\emptyset$,
then $\max\{i,j\}=\max\{k,l\}$. Here is another consequence of (ii): if $a,a'\in I_i$ with $a\not=a'$
and $j,j'<i$, then $(a+\lfloor\la I_j\rfloor)\cap(a'+\lfloor\la I_{j'}\rfloor)=\emptyset$. Indeed,
$|a-a'|\ge\mu_i\gg\la b_{i-1}$. Thus
$$
(E+\lfloor\la F\rfloor)\cap(a_i,a_{i+1}]=\left\{
\begin{array}{ll}
\displaystyle\bigsqcup_{a\in I_i\atop{j=2,4,\ldots,i-1}}(a+\lfloor\la I_j\rfloor)&,\text{ if }i\text{ is odd}\\
&\\
\displaystyle\bigsqcup_{b\in I_i\atop{j=1,3,\ldots,i-1}}(I_j+\lfloor\la b\rfloor)&,\text{ if }i\text{ is even.}\\
\end{array}\right.
$$
Here is a consequence of (v): if $i$ is odd, $a,a'\in I_i$ with $a\not=a'$ and $j,k<i$ are even, then the gap
between $a+\lfloor\la I_j\rfloor$ and $a'+\lfloor\la I_k\rfloor$ has length at least $\frac{|a-a'|}{2}$;
if $i$ is even, $b,b'\in I_i$ with $b\not=b'$ and $j,k<i$ are odd, then the gap between $I_j+\lfloor\la b\rfloor$ and
$I_k+\lfloor\la b'\rfloor$ has length at least $\frac{|b-b'|}{4\la}$.

Now we prove that $d^*(E+\lfloor\la F\rfloor)=0$.
Let $I=(M,N]\subset\Z$ with $M+1,N\in E+\lfloor \la F\rfloor$, say
$$
\left\{
\begin{array}{l}
M+1=a'+\lfloor\la b'\rfloor\in I_k+\lfloor\la I_l\rfloor\\
N  =a+\lfloor\la b\rfloor\in I_i+\lfloor\la I_j\rfloor\,.\\
\end{array}\right.
$$
We have $\max\{i,j,k,l\}=i$ or $j$. Without loss of generality\footnote{The reverse
case is symmetric, because $E+\lfloor\la F\rfloor$ is basically
$\lfloor\la(F+\lfloor\la^{-1}E\rfloor)\rfloor$ and, with this interpretation,
the roles of $I_i$ and $I_j$ are interchanged.}, assume that $i>j$. We have two cases:\\

\noindent {$\bullet$} $a=a'$: $r+\lfloor\la s\rfloor\in(E+\lfloor\la F\rfloor)\cap I$
iff $r=a$ and $s\in F\cap[b',b]$, thus
\begin{equation}\label{eq 23}
\dfrac{|(E+\lfloor\la F\rfloor)\cap I|}{|I|}\le \dfrac{|b-b'+1|^\al}{|\lfloor\la b\rfloor-\lfloor\la b'\rfloor|}
\sim\dfrac{|b-b'|^\al}{\la|b-b'|}=\dfrac{1}{\la|b-b'|^{1-\al}}\,\cdot
\end{equation}

\noindent {$\bullet$} $a>a'$: if $r+\lfloor\la s\rfloor\in(E+\lfloor\la F\rfloor)\cap I$, then
$r\in E\cap[a',a]$ and $s\in I_2\cup I_4\cup\cdots\cup I_{i-1}$, thus
\begin{equation}\label{eq 24}
\dfrac{|(E+\lfloor\la F\rfloor)\cap I|}{|I|}\le \dfrac{b_{i-1}|E\cap[a',a]|}{\frac{|a-a'|}{2}}
\sim\dfrac{b_{i-1}|a-a'|^\al}{|a-a'|}\le\dfrac{b_{i-1}}{{\mu_i}^{1-\al}}\le \dfrac{1}{b_{i-1}}\ ,
\end{equation}
where in the last inequality we used (ii).

By (\ref{eq 23}) and (\ref{eq 24}), it follows that $d^*(E+\lfloor\la F\rfloor)=0$.

\subsection{Another counterexample to Theorem \ref{main thm 1}}\label{sub counterexample 2}

We now prove that the set of parameters in Theorem \ref{main thm 1} cannot be $\mathbb Z$:
we construct $E\subset\Z$ regular such that $D(E)=D(E+\la E)$
for all $\la\in\Z$.

Given $\al\in\left(\frac{1}{2},1\right)$ and $c\in\Z$, let $E(\al,c)$ be the
generalized IP-set associated to the sequences $k_n=c2^n$ and
$d_n=\left\lfloor 2^{n^2/2\al}\right\rfloor$. If $n$ is large enough, then
$$\sum_{i=1}^{n-1}k_i d_i\le c\sum_{i=1}^{n-1}2^{\frac{i^2}{2\al}+i}\le
c\sum_{j=1}^{\frac{(n-1)^2}{2\al}+n-1}2^j<c2^{\frac{(n-1)^2}{2\al}+n}<d_n.$$
By Lemma \ref{lemma gap}, $E(\al,c)$ has counting dimension $\al$. Let $E=E(\al,1)$.
Thus $E+\la E=E(\al,\la+1)$ has counting dimension $\al$.
If $E$ is regular, we are done. If not, we apply Proposition \ref{prop smaller dimension}
to get a regular subset of $E$ yet with counting dimension greater than $\frac{1}{2}$.

\section{Proofs}\label{section marstrand thm}

Fix $E,F\subset\Z$ regular and compatible. Throughout the proof, we fix a compact interval
$\Lambda\subset(0,+\infty)$. Given distinct points $z=(a,b)$ and $z'=(a',b')$ of $E\times F$, let
$$\Lambda_{z,z'}=\left\{\la\in\Lambda:a+\lfloor \la b\rfloor=a'+\lfloor\la b'\rfloor\right\}.$$
Clearly, $\Lambda_{z,z'}$ is empty if $b=b'$.

\begin{lemma}\label{lemma estimate lambda}
If $z=(a,b),z'=(a',b')$ are distinct points of $\Z^2$ and $\Lambda_{z,z'}\not=\emptyset$,
then:
\begin{enumerate}[(a)]\label{eq 5}
\item $m(\Lambda_{z,z'})\lesssim|b-b'|^{-1}$.
\item $|b-b'|\min\Lambda-1\le|a-a'|\le|b-b'|\max\Lambda+1$.
\end{enumerate}
\end{lemma}

\begin{proof}
Assume that $b>b'$, and let $\la\in\Lambda_{z,z'}$, say
$a+\lfloor \la b\rfloor=n=a'+\lfloor\la b' \rfloor$. Thus
$$
\left\{\begin{array}{rcccl}
n-a&\le&\la b&<&n-a+1\\
n-a'&\le&\la b'&<&n-a'+1\\
\end{array}\right.
\ \Longrightarrow\ \ a'-a-1<\la(b-b')<a'-a+1,
$$
so $\Lambda_{z,z'}\subset\left(\frac{a'-a-1}{b-b'}\ ,\frac{a'-a+1}{b-b'}\right)$,
which proves (a). Also:
$$\dfrac{a'-a-1}{b-b'}\le\min\Lambda_{z,z'}\le\max\Lambda\ \ \Longrightarrow\ \ a'-a\le(b-b')\max\Lambda+1$$
and
$$\dfrac{a'-a+1}{b-b'}\ge\max\Lambda_{z,z'}\ge\min\Lambda\ \ \Longrightarrow\ \ a'-a\ge\min(b-b')\Lambda-1.$$
\end{proof}

Lemma \ref{lemma estimate lambda} expresses the crucial property of transversality that is present
in most results related to Marstrand's theorem. By (b), if $\Lambda_{z,z'}\not=\emptyset$ then
$|a-a'|\sim|b-b'|$ .

Let $(I_n)_{n\ge 1}$ and $(J_n)_{n\ge 1}$ be sequences of intervals satisfying Definition \ref{def compatible}.
For each pair $(n,\la)\in\N\times\Lambda$, let
$$N_n(\la)=\left\{((a,b),(a',b'))\in ((E\cap I_n)\times(F\cap J_n))^2:
a+\lfloor\la b\rfloor=a'+\lfloor\la b'\rfloor\right\},$$
and let $\Delta_n=\displaystyle\int_\Lambda |N_n(\la)|dm(\la)$.
By a double counting argument,
\begin{equation}\label{eq 7}
\Delta_n=\sum_{z,z'\in(E\cap I_n)\times(F\cap J_n)}m(\Lambda_{z,z'})\,.
\end{equation}

\begin{lemma}\label{estimate of integral}
Denote $D(E)=\al$ and $D(F)=\be$.
\begin{enumerate}[(a)]
\item If $\al+\be<1$, then $\Delta_n\lesssim|I_n|^{\al+\be}$.
\vspace{.3cm}
\item If $\al+\be>1$, then $\Delta_n\lesssim|I_n|^{2\al+2\be-1}$.
\end{enumerate}
\end{lemma}

\begin{proof}
By (\ref{eq 7}),
\begin{eqnarray*}
\Delta_n&=&\sum_{z,z'\in(E\cap I_n)\times(F\cap J_n)}m(\Lambda_{z,z'})\\
    &=&\sum_{a\in E\cap I_n\atop{b\in F\cap J_n}}\sum_{s=1}^{\log |I_n|}\sum_{a'\in E\cap I_n\atop{|a-a'|\sim e^s}} \sum_{b'\in F\cap J_n\atop{|b-b'|\sim e^s}}m(\Lambda_{z,z'})\\
    &\lesssim&\sum_{a\in E\cap I_n\atop{b\in F\cap J_n}}\sum_{s=1}^{\log |I_n|}e^{-s}(e^{s})^\al(e^{s})^\be\\
    &=&\sum_{a\in E\cap I_n\atop{b\in F\cap J_n}}\sum_{s=1}^{\log |I_n|}(e^s)^{\al+\be-1}\\
    &\lesssim& |I_n|^{\al+\be}\sum_{s=1}^{\log |I_n|}\left(e^{\al+\be-1}\right)^s,
\end{eqnarray*}
thus
$$
\Delta_n\lesssim\left\{
\begin{array}{ll}
|I_n|^{\al+\be}|I_n|^{\al+\be-1}=|I_n|^{2\al+2\be-1}&\text{, if }\al+\be>1,\\
&\\
|I_n|^{\al+\be}&\text{, if }\al+\be<1.
\end{array}\right.
$$
\end{proof}

\subsection{Proof of Theorem \ref{main thm 1}}
We divide the proof into three parts.\\

\noindent {\bf Part 1.} $\al+\be<1$: fix $\ve>0$ and $n\ge 1$. By Lemma \ref{estimate of integral}, the
set of parameters $\la\in\Lambda$ such that
\begin{equation}\label{eq 9}
|N_n(\la)|\lesssim\dfrac{|I_n|^{\al+\be}}{\ve}
\end{equation}
has Lebesgue measure at least $m(\Lambda)-\ve$. We will prove that
\begin{equation}\label{eq 10}
\dfrac{|(E+\lfloor\la F\rfloor)\cap(I_n+\lfloor\la J_n\rfloor)|}{|I_n+\lfloor\la J_n\rfloor|^{\al+\be}}\gtrsim \ve
\end{equation}
for every $\la\in\Lambda$ satisfying (\ref{eq 9}). For $(m,n,\la)\in\Z\times\Z\times\Lambda$, let
$$s(m,n,\la)=|\left\{(a,b)\in(E\cap I_n)\times(F\cap J_n):a+\lfloor\la b\rfloor =m\right\}|\,.$$
Thus
\begin{equation}\label{eq 28}
\sum_{m\in\Z} s(m,n,\la)=|E\cap I_n||F\cap J_n|\sim |I_n|^{\al+\be}
\end{equation}
and
\begin{equation}\label{eq 29}
\sum_{m\in\Z} s(m,n,\la)^2=|N_n(\la)|\lesssim \dfrac{|I_n|^{\al+\be}}{\ve}\,\cdot
\end{equation}
The numerator in (\ref{eq 10}) is at least the cardinality of the set $S(n,\la)=\{m\in\Z:s(m,n,\la)>0\}$,
because $(E+\lfloor\la F\rfloor)\cap(I_n+\lfloor\la J_n\rfloor)$ contains $S(n,\la)$.
By the Cauchy-Schwarz inequality and (\ref{eq 28}), (\ref{eq 29}), we have
$$|S(n,\la)|\ge\dfrac{\left(\displaystyle\sum_{m\in\Z} s(m,n,\la)\right)^2}{\displaystyle\sum_{m\in\Z}s(m,n,\la)^2}
\gtrsim\dfrac{\left(|I_n|^{\al+\be}\right)^2}{\dfrac{|I_n|^{\al+\be}}{\ve}}
=\ve|I_n|^{\al+\be}.
$$
Because $|I_n+\lfloor\la J_n\rfloor|\sim |I_n|$, we get that
$$\dfrac{|(E+\lfloor\la F\rfloor)\cap(I_n+\lfloor\la J_n\rfloor)|}{|I_n+\lfloor\la J_n\rfloor|^{\al+\be}}
\gtrsim\dfrac{|S(n,\la)|}{|I_n|^{\al+\be}}\gtrsim\ve\,,$$
establishing (\ref{eq 10}).

For each $n\ge 1$, let $G_\ve^n=\{\la\in\Lambda:(\ref{eq 10})\text{ holds}\}$.
Then $m\left(\Lambda\backslash G_\ve^n\right)\le\ve$, and the same holds for
$G_\ve=\bigcap_{n\ge 1}\bigcup_{l=n}^{\infty}G_\ve^l$.
For each $\la\in G_\ve$, $H_{\al+\be}(E+\lfloor\la F\rfloor)>0$, thus $D(E+\lfloor\la F\rfloor)\ge\al+\be$.
Because $G=\bigcup_{n\ge 1}G_{1/n}\subset\Lambda$ has Lebesgue measure $m(\Lambda)$,
Part 1 is complete.\\

\noindent {\bf Part 2.} $\al+\be>1$: for a fixed $\ve>0$,
Lemma \ref{estimate of integral} implies that the set of parameters $\la\in\Lambda$ such that
$|N_n(\la)|\lesssim\frac{|I_n|^{2\al+2\be-1}}{\ve}$ has Lebesgue measure at least $m(\Lambda)-\ve$.
In this case,
$$
|S(n,\la)|\ge\dfrac{\left(\displaystyle\sum_{m\in\Z} s(m,n,\la)\right)^2}{\displaystyle\sum_{m\in\Z}s(m,n,\la)^2}
\gtrsim\dfrac{\left(|I_n|^{\al+\be}\right)^2}{\dfrac{|I_n|^{2\al+2\be-1}}{\ve}}
=\ve|I_n|
$$
thus
$$\dfrac{|(E+\lfloor\la F\rfloor)\cap(I_n+\lfloor\la J_n\rfloor)|}{|I_n+\lfloor\la J_n\rfloor|}
\gtrsim\dfrac{|S(n,\la)|}{|I_n|}\gtrsim\ve.$$
The Borel-Cantelli argument is analogous to Part 1.\\

\noindent {\bf Part 3.} $\al+\be=1$: let $n\ge 1$. Because $E$ is regular, there is
$E_n\subset E$, regular and compatible
with $F$ such that $D(E)-\frac{1}{n}<D(E_n)<D(E)$. Thus $1-\frac{1}{n}<D(E_n)+D(F)<1$.
By Part 1, there is a set $\Lambda_n$ of full Lebesgue measure such that
$D(E_n+\lfloor\la F\rfloor)\ge 1-\frac{1}{n}$ for all $\la\in\Lambda_n$.
The set $\Lambda=\bigcap_{n\ge 1}\Lambda_n$ has full Lebesgue measure and
$D(E+\lfloor\la F\rfloor)\ge 1$ for all $\la\in\Lambda$.

\subsection{Proof of Theorem \ref{main thm 2}} We divide it into two parts.\\

\noindent{\bf Part 1.} $\sum_{i=0}^k D(E_i)\le 1$: by Theorem \ref{main thm 1},
$$D(E_0+\lfloor\la_1 E_1\rfloor)\ge D(E_0)+D(E_1)\ ,\ \ \ m\text{-a.e. }\la_1\in\R.$$
To  each of these parameters, apply Proposition \ref{prop smaller dimension} to obtain a
regular subset $F_{\la_1}\subset E_0+\lfloor \la_1 E_1\rfloor$ such that
$D(F_{\la_1})=D(E_0)+D(E_1)$. Because $E_2$ is universal, we can apply Theorem \ref{main thm 1}
again and get that
$$D(F_{\la_1}+\lfloor\la_2 E_2\rfloor)\ge D(E_0)+D(E_1)+D(E_2)\,,
\ \ \ m\text{-a.e. }\la_2\in\R.$$
By Fubini's theorem,
$$D(E_0+\lfloor\la_1 E_1\rfloor+\lfloor\la_2 E_2\rfloor)\ge D(E_0)+D(E_1)+D(E_2)\,,
\ \ \ m_2\text{-a.e. }(\la_1,\la_2)\in\R^2.$$
Iterating the above arguments, it follows that
$$D(E_0+\lfloor\la_1 E_1\rfloor+\cdots+\lfloor\la_k E_k\rfloor)\ge
D(E_0)+\cdots+D(E_k)\,,\ \ \ m_k\text{-a.e. }(\la_1,\ldots,\la_k)\in\R^k.$$

\vspace{.2cm}
\noindent{\bf Part 2.} $\sum_{i=0}^k D(E_i)>1$: without loss of generality, we can assume that
$$D(E_0)+\cdots+D(E_{k-1})\le 1<D(E_0)+\cdots+D(E_{k-1})+D(E_k)\,.$$
By Part 1,
$$D(E_0+\lfloor\la_1 E_1\rfloor+\cdots+\lfloor\la_{k-1}E_{k-1}\rfloor)\ge
D(E_0)+\cdots+D(E_{k-1})$$
for $m_{k-1}\text{-a.e. }(\la_1,\ldots,\la_{k-1})\in\R^{k-1}$.
To each of these $(k-1)$-tuples, let
$F_{(\la_1,\ldots,\la_{k-1})}$ $\subset E_0+\cdots+\lfloor \la_{k-1}E_{k-1}\rfloor$ regular with
$D\left(F_{(\la_1,\ldots,\la_{k-1})}\right)=D(E_0)+\cdots+D(E_{k-1})$.
Because $D(F_{(\la_1,\ldots,\la_{k-1})})+D(E_k)>1$, Theorem \ref{main thm 1} gives
$d^*\left(F_{(\la_1,\ldots,\la_{k-1})}+\lfloor\la_k E_k\rfloor\right)>0$ for $m\text{-a.e. }\la_k\in\R$.
By Fubini's theorem, the proof is complete.

\section{Concluding remarks}\label{section final remarks}

We think there is a more specific way of defining the counting dimension that encodes the conditions
of regularity and compatibility. A natural candidate would be a prototype of a Hausdorff dimension,
where one looks to all covers, properly renormalized in the unit interval, and takes a $\liminf$.
An alternative definition appeared in \cite{N}. It would be a natural program to prove
Marstrand type results in this context.

Another interesting question is to consider arithmetic sums $E+\la F$, where $\la\in\Z$. These are genuine
arithmetic sums and, as we saw in Section\ref{sub counterexample 2}, their dimension may not increase.
We think very strong conditions on the sets $E,F$ are needed to prove analogous results about $E+\la F$
for $\la\in\Z$.

We also think the results obtained here also work to subsets of $\Z^k$. Given $E\subset\Z^k$, the
{\it upper Banach density} of $E$ is equal to
$$d^*(E)=\limsup_{|I_1|,\ldots,|I_k|\rightarrow\infty}\dfrac{|E\cap(I_1\times\cdots\times I_k)|}{|I_1\times\cdots\times I_k|}\,,$$
where $I_1,\ldots,I_k$ run over all intervals of $\Z$, the {\it counting dimension} of $E$ is
$$D(E)=\limsup_{|I_1|,\ldots,|I_k|\rightarrow\infty}
\dfrac{\log|E\cap(I_1\times\cdots\times I_k)|}{\log|I_1\times\cdots\times I_k|}\,,$$
where $I_1,\ldots,I_k$ run over all intervals of $\Z$ and, for $\al\ge 0$, the {\it counting $\al$-measure}
of $E$ is
$$H_\al(E)=\limsup_{|I_1|,\ldots,|I_k|\rightarrow\infty}\dfrac{|E\cap(I_1\times\cdots\times I_k)|}{|I_1\times\cdots\times I_k|^\al}\,,$$
where $I_1,\ldots,I_k$ run over all intervals of $\Z$. These quantities satisfy similar properties to those
in Section \ref{sub definitions}. The notion of regularity is defined in an analogous manner.
For compatibility, we take into account the geometry of $\Z^k$.
Two regular subsets $E,F\subset\Z^k$ are compatible if there exist sequences of
boxes $R_n=I_1^n\times\cdots\times I_k^n$ and $S_n=J_1^n\times\cdots\times J_k^n$ such that
\begin{enumerate}[(i)]
\item $|I_i^n|\sim |J_i^n|$ for every $i=1,2,\ldots,k$, and
\item $|E\cap R_n|\gtrsim |R_n|^{D(E)}$ and $|F\cap S_n|\gtrsim |S_n|^{D(F)}$.
\end{enumerate}
We think the theory developed in this article can be extended to prove that: if $E,F\subset\Z^k$ are two
regular compatible subsets, then
$D(E+\lfloor\la F\rfloor)\ge\min\{1,D(E)+D(F)\}$ for Lebesgue almost every $\la\in\R$; if
$D(E)+D(F)>1$, then $E+\lfloor\lambda F\rfloor$ has positive upper Banach density for Lebesgue almost
every $\lambda\in\R$.

\section*{Acknowledgments}

The authors are thankful to IMPA for the excellent ambient during the preparation of this
manuscript, to Simon Griffiths and Rob Morris for the conversations that gave rise to the counterexample
of Section \ref{sub counterexample 2}, and to the referee for many useful and detailed suggestions.
Y.L. is also grateful to Enrique Pujals for his constant encouragement. During the preparation of this
manuscript, Y.L. was a doctoral student at IMPA. This research was possible due to the support of
J. Palis 2010 Balzan Prize for Mathematics, CNPq-Brazil and Faperj-Brazil.
Y.L. is supported by the Brin Fellowship.

\bibliographystyle{amsplain}

\end{document}